\documentclass[a4paper, 12 pt]{article}
\usepackage{amscd}
\usepackage{amssymb}
\usepackage{amsfonts}
\usepackage{amsmath}
\allowdisplaybreaks
\usepackage{amsthm}
\usepackage{bbm}
\usepackage{CJK}
\usepackage{fancyhdr}
\usepackage{graphicx}

\usepackage{geometry}
\usepackage{ifpdf}
\ifpdf
  \usepackage[colorlinks=true,citecolor=black,linkcolor=blue,final,hyperindex]{hyperref}
\else
  \usepackage[colorlinks,final,backref=page,hyperindex,hypertex]{hyperref}
\fi
\usepackage{tikz}
\usepackage[active]{srcltx}

\usepackage{indentfirst}
\usepackage{latexsym}
\usepackage{mathrsfs}
\usepackage{xypic}

\usepackage{pxfonts}

\renewcommand{\theequation}{\arabic{section}.\arabic{equation}}
\def\vbar{\mathchoice{\vrule height6.3ptdepth-.5ptwidth.8pt\kern- .8pt}
{\vrule height6.3ptdepth-.5ptwidth.8pt\kern-.8pt} {\vrule
height4.1ptdepth-.35ptwidth.6pt\kern-.6pt} {\vrule
height3.1ptdepth-.25ptwidth.5pt\kern-.5pt}}

\newcommand{\Hom}{\rm Hom}

\newcommand{\Chom}{\rm Chom}

\newcommand{\C}{{\rm C}}
\newcommand{\QC}{{\rm QC}}
\newcommand{\GDer}{{\rm GDer}}
\newcommand{\ZDer}{{\rm ZDer}}
\newcommand{\Cend}{{\rm Cend}}
\newcommand{\QDer}{{\rm QDer}}
\newcommand{\CDer}{{\rm CDer}}
\newcommand{\Der}{{\rm Der}}

\newcommand{\R}{\mathcal{R}}

\newcommand{\g}{\mathfrak{g}}

\def\<{\langle}
\def\>{\rangle}
\def\a{\alpha}
\def\b{\beta}
\def\f{\phi}
\def\p{\psi}

\def\d{\delta}

\def\g{\gamma}

\def\l{\lambda}

\def\ra{\rightarrow}

\def\si{\sigma}
\def\th{\theta}

\def\t{\tau}

\def\vp{\varphi}

\def\pr{\partial}
\def\R{\mathcal{R}}

\theoremstyle{definition}
\newtheorem{df}{Definition}[section]
\theoremstyle{plain}
\newtheorem{thm}[df]{Theorem}
\newtheorem{cor}[df]{Corollary}

\newtheorem{prop}[df]{Proposition}
\newtheorem{lem}[df]{Lemma}

\theoremstyle{definition}

  \theoremstyle{definition}
    \newtheorem{exa}[df]{Example}

\newcommand{\bmx}{\begin{pmatrix}}
\newcommand{\emx}{\end{pmatrix}}
\date{}

\begin{document}
\title{ \bf  Conformal $(\si,\t)$-Derivations on Lie conformal superalgebras }
\author{ Tianqi Feng$^{1}$, Jun Zhao$^{2}$, Liangyun Chen$^{1}$
 \date{{\small $^{1}$School of Mathematics and Statistics, Northeast Normal
 University,\\ Changchun 130024, China\\
$^{2}$School of Mathematics and Statistics, Henan University, Kaifeng 475004, China}}}
 \maketitle
\begin{abstract}
In this paper, we focus on the $(\si,\t)$-derivation theory of  Lie conformal superalgebras. Firstly, we study the fundamental properties of conformal $(\si,\t)$-derivations. Secondly, we mainly research the interiors of conformal $G$-derivations. Finally, we discuss the relationships between the conformal $(\si,\t)$-derivations and some generalized conformal derivations of Lie conformal superalgebras.
\end{abstract}

$(\a,\b,\g)$-derivation.
\renewcommand{\thefootnote}{\fnsymbol{footnote}}
\footnote[0]{ Corresponding author(L. Chen): chenly640@nenu.edu.cn.}
\footnote[0]{Supported by  NNSF of China (Nos. 11771069 and 12071405),  China Postdoctoral Science Foundation (2020M682272) and NSF of Hennan Province (212300410120).}

\tableofcontents
\numberwithin{equation}{section}
\section*{Introduction}
\def\theequation{0. \arabic{equation}}
\setcounter{equation} {0}
Lie conformal superalgebras, introduced by Kac in \cite{Kac98,Kac99}, encodes the singular part of the operator product expansion of chiral fields in two-dimensional quantum field theory. Furthermore, the category of Lie conformal superalgebras $\R$ is equivalent to the category of formal distribution Lie superalgebras $(Lie\R,\R)$, which is essentially an infinite-dimensional Lie superalgebras. Namely, they are closely connected to the notion of a formal distribution Lie superalgebra $(g,\mathcal{F})$, which means that a Lie superalgebra $g$ spanned by the coefficients of a family $\mathcal{F}$ of mutually local formal distributions. See \cite{FKR2004} in details.

The derivation theory of Lie conformal (super)algebras was introduced in \cite{DK1998,FKR2004}. Lately, the generalized derivation theory of Lie conformal (super)algebras was developed in \cite{FHS2019,ZCY2017,ZYC2018}. \cite{NH2008,ZZ2013} studied the $(\a,\b,\g)$-derivations of Lie (super)algebras. \cite{CHL2020} studied a kind of new generalized derivations of Lie algebras, that is the $(\si,\t)$-derivation theory of Lie algebras. In the present paper, we aim to do same in \cite{CHL2020,FHS2019} for Lie conformal superalgebras, extend the $(\si,\t)$-derivations of Lie algebras and conformal $(\a,\b,\g)$-derivations of Lie conformal algebras to the Lie conformal superalgebras.

This paper is organized as follows. In Section $1$, we recall several basic definitions of Lie conformal superalgebras and introduce the concept of conformal $(\si,\t)$-derivations of Lie conformal superalgebras. In Section $2$, we obtain some fundamental properties of conformal $(\si,\t)$-derivations. In Section $3$, we describe the interiors of conformal $G$-derivations and compute the corresponding Hilbert series to show its complexity. In Section $4$, we devote to study the connection between the conformal $(\si,\t)$-derivations and some generalized conformal derivations. What's more, we introduce the concept of conformal $(\a,\b,\g)$-derivations of Lie conformal superalgebras and obtain some connection between the generalized conformal derivations and the conformal $(\a,\b,\g)$-derivations.


\section{Preliminaries}
\def\theequation{\arabic{section}.\arabic{equation}}
\setcounter{equation} {0}

Let $V$ be a superspace that is a $\mathbb{Z}_2$-graded linear space with a direct sum
$V = V_{\bar0} \oplus V_{\bar1}$. The elements of $V_j$, $j = \{\bar0, \bar1\}$, are said to be homogenous and of parity $j$.
The parity of
a homogeneous element $x$ is denoted by $|x|$. Throughout what follows, if $|x|$ occurs
in an expression, then it is assumed that $x$ is homogeneous and that the expression
extends to the other elements by linearity.

 The following notion was due to \cite{Kac98}.

\begin{df}\cite{Kac98}
A \emph{Lie conformal superalgebra} $\R$ is a left $\mathbb{Z}_{2}$-graded $\mathbb{C}[\partial]$-module, and for any $n\in\mathbb{Z}_{\geq0}$ there is a
family of $\mathbb{C}$-linear  $n$-products from $\R\otimes\R$ to $\R$ satisfying the following conditions
\begin{enumerate}
\item[$(\rm C0)$]  For any $a,b\in\R$, $a_{(n)}b=0$ for $n\gg 0$,
\item[$(\rm C1)$]  For any $a,b\in\R$ and $n\in\mathbb{Z}_{\geq0}$, $(\partial a)_{(n)}b=-n a_{(n-1)}b$,
\item[$(\rm C2)$]  For any $a,b\in\R$ and $n\in\mathbb{Z}_{\geq0}$, $$a_{(n)}b=-(-1)^{|a||b|}\sum_{j=0}^{\infty}(-1)^{j+n}\frac{1}{j!}\partial^j(b_{(n+j)}a),$$
\item[$(\rm C3)$]  For any $a,b,c\in\R$ and $m,n\in\mathbb{Z}_{\geq0}$,
$$a_{(m)}(b_{(n)}c)=\sum_{j=0}^{m}(_{j}^{m})(a_{(j)}b)_{(m+n-j)}c+(-1)^{|a||b|}b_{(n)}(a_{(m)}c).$$
\end{enumerate}
(Convention: $a_{(n)}b = 0$ if  $n < 0$). Note that if we define $\lambda$-bracket $[-_{\lambda}-]$:
\begin{eqnarray*}
[a_\lambda b]=\sum_{n=0}^{\infty}\frac{\lambda^n}{n!}a_{(n)}b,\,\forall a,b\in\R.
\end{eqnarray*}
That is, $\R$ is a Lie conformal superalgebra if and only if $[-_{\lambda}-]$ satisfies the following axioms
\begin{eqnarray*}
&&(\rm C1)_{\lambda}\ \ \ \ {\rm  Conformal\ sesquilinearity}:\ \, [(\partial a)_\lambda b]=-\lambda[a_\lambda b];\\
&&(\rm C2)_{\lambda}\ \ \ \ {\rm Skew-supersymmetry}: \ \, [a_\lambda b]=-(-1)^{|a||b|}[b_{-\partial-\lambda}a];\\
&&(\rm C3)_{\lambda}\ \ \ \ {\rm Jacobi\ identity}:  \ \, [a_\lambda[b_\mu c]]=[[a_\lambda b]_{\lambda+\mu}c]+(-1)^{|a||b|}[b_\mu[a_\lambda c]].
\end{eqnarray*}
\end{df}
A Lie conformal superalgebras are called \emph{finite} if $\R$ is a finitely generated $\mathbb{C}[\pr]$-module. The \emph{rank} of a conformal algebra $\R$ is its rank as a $\mathbb{C}[\pr]$-module (recall that this is the dimension over $\mathbb{C}(\pr)$, the field of fractions of $\mathbb{C}[\pr]$, of $\mathbb{C}(\pr)\otimes_{\mathbb{C}[\pr]}\R$).

Throughout this paper, we assume that $\R$ is finite.

\begin{df}\cite{Kac98}
An \emph{associative conformal superalgebra} $\R$ is a left $\mathbb{Z}_{2}$-graded $\mathbb{C}[\partial]$-module endowed with a $\lambda$-product
from $\R\otimes\R$ to $\mathbb{C}[\lambda]\otimes\R$, for any $a,b,c\in\R$,
satisfying the following conditions:
\begin{enumerate}
\item [$(1)$] $(\partial a)_\lambda b=-\lambda a_\lambda b,\ a_\lambda(\partial b)=(\partial+\lambda)(a_\lambda b)$,
\item [$(2)$] $a_\lambda(b_\mu c)=(a_\lambda b)_{\lambda+\mu}c$.
\end{enumerate}
\end{df}

\begin{df}\cite{FKR2004}
Let $M$ and $N$ be $\mathbb{Z}_{2}$-graded $\mathbb{C}[\partial]$-modules. A \emph{conformal linear map} of degree $\theta$ from $M$ to $N$ is a sequence $f=\{f_{(n)}\}_{n\in\mathbb{Z}_{\geq0}}$ of $f_{(n)}\in\Hom_{\mathbb{C}}(M,N)$ satisfying that
 $$\partial_{N} f_{(n)}-f_{(n)}\partial_{M}=-n f_{(n-1)},\ n\in\mathbb{Z}_{\geq0} \ \ \ {\rm  and} \ \ \
f_{(n)}(M_{\mu})\subseteq N_{\mu+\theta},  \ \mu,\theta\in\mathbb{Z}_{2}.$$
 Set $f_\lambda=\sum_{n=0}^{\infty}\frac{\lambda^n}{n!}f_{(n)}$. Then
 $f=\{f_{(n)}\}_{n\in\mathbb{Z}_{\geq0}}$ is a conformal linear map of degree $\theta$ if and only if
 $$f_{\lambda}\partial_{M}=(\partial_{N}+\lambda) f_{\lambda} \ \ \ {\rm and} \ \ \
f_{\lambda}(M_{\mu})\subseteq N_{\mu+\theta}[\lambda],  \ \mu,\theta\in\mathbb{Z}_{2}.$$
\end{df}

Let $\Chom(M,N)_{\theta}$ denote the set of conformal linear maps of degree $\theta$ from $M$ to $N$. Then $\Chom(M,N)=\Chom(M,N)_{\bar{0}}\oplus\Chom(M,N)_{\bar{1}}$ is a $\mathbb{Z}_{2}$-graded $\mathbb{C}[\partial]$-module via:
\begin{eqnarray*}\partial f_{(n)}=-n f_{(n-1)},\ {\rm equivalently},\ \partial f_{\lambda}=-\lambda f_{\lambda}.\end{eqnarray*}
The composition $f_{\lambda}g:L\rightarrow N\otimes\mathbb{C}[\lambda]$ of conformal linear maps $f:M\rightarrow N$ and $g:L\rightarrow M$ is given by
 \begin{eqnarray*}
(f_{\lambda}g)_{\lambda+\mu}=f_{\lambda}g_{\mu},\ \ \ \forall \, f,g\in \Chom(M,N).
 \end{eqnarray*}

If $M$ is a finitely generated $\mathbb{Z}_{2}$-graded $\mathbb{C}[\partial]$-module, then $\Cend(M):=\Chom(M,M)$ is an associative conformal superalgebra with respect to the above composition. Thus, $\Cend(M)$ becomes a Lie conformal superalgebra, called the general linear Lie conformal superalgebra, denoted as $gc(M)$, with respect to the following $\lambda$-bracket(see \cite[Example1.1]{FKR2004}):
\begin{eqnarray}\label{def2-5}
[f_{\lambda}g]_{\mu}=f_{\lambda}g_{\mu-\lambda}-(-1)^{|f||g|}g_{\mu-\lambda}f_{\lambda}.
\end{eqnarray}
Hereafter all $\mathbb{Z}_{2}$-graded $\mathbb{C}[\partial]$-modules are supposed to be  finitely generated.

\begin{df}\cite{FKR2004}
Let $\R$ be a Lie conformal superalgebra. $d\in\Cend(\R)$ is a \emph{conformal derivation} of degree $\th$ if for any $a,b\in\R$ it holds that
\begin{eqnarray*}
d_{(m)}(a_{(n)}b)=\sum_{j=0}^{m}(_{j}^{m})(d_{(j)}a)_{(m+n-j)}b+(-1)^{|a||\th|}a_{(n)}(d_{(m)}(b));
\end{eqnarray*}
equivalently,
\begin{eqnarray*}
d_\lambda([a_\mu b])=[(d_\lambda(a))_{\lambda+\mu}b]+(-1)^{|a||\th|}[a_\mu(d_\lambda(b))].\label{defa}
\end{eqnarray*}
\end{df}
For any $r\in\R$, $d^{r}_{\lambda}$ is called an \emph{inner} conformal derivation of $\R$ if $d^{r}_{\lambda}(r')=[r_{\lambda}r']$, $\forall r'\in\R$.

Define $\CDer(\R)_{\it{\th}}$ is the set of conformal derivations of degree $\th$ of $\R$, and $\CDer(\R)=\rm{CDer}_{\bar{0}}(\R)\oplus \CDer_{\bar{1}}(\R)$. Then it is obvious that $\CDer(\R)$ is a subalgebra of $\Cend(\R)$.

\begin{df}\cite{ZCY2017}
An element $f$ in $\Cend(\R)_\theta$ is called
\begin{itemize}
\item a {\it generalized derivation} of degree $\theta$ of $\R$, if there exist $f^{'},f^{''}\in\Cend(\R)_\theta$ such that
\begin{eqnarray*}\label{5-1}
[(f_\lambda(a))_{\lambda+\mu}b]+(-1)^{\theta|a|}[a_\mu(f^{'}_\lambda(b))]=f^{''}_\lambda([a_\mu b]), \ \forall \ a,b\in\R.
\end{eqnarray*}
\item a {\it quasiderivation }of degree $\theta$ of $\R$, if there is $f^{'}\in\Cend(\R)_\theta$ such that
\begin{eqnarray*}\label{5-2}
[(f_\lambda(a))_{\lambda+\mu}b]+(-1)^{\theta|a|}[a_\mu(f_\lambda(b))]=f^{'}_\lambda([a_\mu b]), \ \forall \ a,b\in\R.
\end{eqnarray*}
\item a {\it centroid} of degree $\theta$ of $\R$, if it satisfies
\begin{eqnarray*}\label{5-3}
 [(f_\lambda(a))_{\lambda+\mu}b]=(-1)^{\theta|a|}[a_\mu(f_\lambda(b))]=f_\lambda([a_\mu b]), \ \forall \ a,b\in\R.
\end{eqnarray*}
\item a {\it quasicentroid} of degree $\theta$ of $\R$, if it satisfies
\begin{eqnarray*}\label{5-4}
[(f_\lambda(a))_{\lambda+\mu}b]=(-1)^{\theta|a|}[a_\mu(f_\lambda(b))], \ \forall \ a,b\in\R.
\end{eqnarray*}
\item a {\it central derivation} of degree $\theta$ of $\R$, if it satisfies
\begin{eqnarray*}\label{5-5}
[(f_\lambda(a))_{\lambda+\mu}b]=f_\lambda([a_\mu b])=0, \ \forall \ a,b\in\R.
\end{eqnarray*}
\end{itemize}
\end{df}

Denote by $\GDer(\R)_\theta$, $\QDer(\R)_\theta$, $\C(\R)_\theta$, $\QC(\R)_\theta$ and $\ZDer(\R)_\theta$ the sets of all generalized derivations, quasiderivations, centroids, quasicentroids and central derivations of degree $\theta$ of $\R$.

\begin{df}\cite{Kac97}
Let $\R,\R^{'}$ be two Lie conformal superalgebras. A \emph{homomorphism} $\f$ from $\R$ to $\R^{'}$ of Lie conformal superalgebras is a homogeneous $\mathbb{C}[\partial]$-linear homomorphism if for any $a,b\in\R$ it holds that
$$\f(a_{(n)}b)=(\f(a)_{(n)}\f(b));$$
equivalently,
$$\f[a_{\l}b]=[\f(a)_{\l}\f(b)].$$
\end{df}

We call $\f$ an isomorphism if it is bijective. We call $\f$ an endomorphism if $\R=\R^{'}$. We call $\f$ an automorphism if it is bijective and if $\R=\R^{'}$.

In the following, we denote $\rm{Aut}(\R)$ the automorphism group of $\R$.

\begin{df}
Suppose that $\R$ is a Lie conformal superalgebra and $G$ is a subgroup of $\rm{Aut}(\R)$. Then $d\in\Cend(\R)$ is a \emph{conformal $G$-derivation} of degree $\th$ of $\R$ if there exist two element $\si,\t$ in $G$ such that
$$d_\l([a_\mu b])=[(d_\l(a))_{\l+\mu}(\si(b))]+(-1)^{|a||\th|}[(\t(a))_\mu (d_\l(b))], \quad \forall a, b \in\R.$$
In this case, $\si$ and $\t$ are called the \emph{associated} automorphisms of $d$.
\end{df}

Denote by $\CDer_G(\R)$ the set of all conformal $G$-derivations of $\R$. It is clear that $\CDer_G(\R)=\CDer(\R)$ if $G$ is a trivial group, that is $G=\left\{\rm{id}_\R\right\}$. Thus, conformal $G$-derivations can be viewed as a generalization of conformal derivations. What's more, if $G\le H$ are two subgroups of $\rm{Aut}(\R)$, then we have $\CDer_G(\R) \subseteq \CDer_H(\R)$ and $\CDer(\R)$ is contained in $\CDer_G(\R)$ for any subgroup $G$ of $\rm{Aut}(\R)$.

Fix two automorphisms $\si,\t \in G$, we denote by $\CDer_{\si,\t}(\R)$ the set of all conformal $G$-derivations associated to $\si$ and $\t$, called the \emph{conformal $(\si,\t)$-derivation}.
It is clear that $\CDer_{\si,\t}(\R)\subseteq \CDer_G(\R)$ is a $\mathbb{C}[\partial]$-module and $\CDer_{\rm{id}_\R,\rm{id}_\R}(\R)=\CDer(\R)$. For convenience, we denote $\CDer_{\si,\rm{id}_\R}(\R)$ by $\CDer_\si(\R)$.

In the following, $G$ always denotes a subgroup of $\rm{Aut}(\R)$.

\section{Fundamental properties}
\def\theequation{\arabic{section}.\arabic{equation}}
\setcounter{equation} {0}

In this section, we aim to show several fundamental properties of conformal $(\si,\t)$-derivations.
\begin{prop}\label{prop:2.1}
Suppose that $\R$ is a Lie conformal superalgebra and  $\si, \t$ are two elements in $G$, then $\rm{rank}(CDer_{\it{\si, \t}}(\R)) = \rm{rank}(CDer_{\it{\t^{-\rm{1}}\si}}(\R))$.
\end{prop}

\begin{proof}
Define a map $\vp_{\tau} : \CDer_{\it{\si, \t}}(\R) \ra \CDer_{\it{\t^{-\rm{1}}\si}}(\R)$ by
$$\vp_{\t}(d) = \t^{-1}d, \quad \forall d \in \CDer_{\it{\si, \t}}(\R).$$
For any $a,b\in\R$, we have
\begin{align*}
&~~~~\t^{-1}(d_\l([a_\mu b]))= \t^{-1}([(d_\l(a))_{\l+\mu}(\si(b))] + (-1)^{|a||d|}[(\t(a))_\mu (d_\l(b))])\\
&= [(\t^{-1}(d_\l(a)))_{\l+\mu}(\t^{-1}(\si(b)))] + (-1)^{|a||d|}[a_\mu(\t^{-1}(d_\l(b)))],
\end{align*}
which means that $\t^{-1}d \in \CDer_{\it{\t^{-\rm{1}}\si}}(\R)$, i.e., the map $\vp_{\t}$ is well-defined.
What's more
$$\vp_\t(d_1+d_2)=\t^{-1}(d_1+d_2)=\t^{-1}d_1+\t^{-1}d_2=\vp_\t d_1+\vp_\t d_2,$$
$$\vp_\t(\partial d)=\vp_\t(-\l d_\l)=-\l\vp_\t(d_\l)=-\l\t^{-1}(d_\l)=\partial\vp_\t(d_\l).$$
That is to say, $\vp_\t$ is a $\mathbb{C}[\partial]$-module homomorphism.

In the following, it still needs to show that $\vp_\t$ is an isomorphism. So we try to see its inverse. We can define a map $\p_\t: \CDer_{\it{\t^{-1}\si}}(\R)\ra \CDer_{\it{\si,\t}}(\R)$ by $\p_\t(d)=\t d$ for any $d$ in $\CDer_{\it{\t^{-1}\si}}(\R)$. Similarly, we can verify that $\p_\t$ is a well defined $\mathbb{C}[\partial]$-module homomorphism.
What's more, we have $\p_\t\vp_\t=\rm{id}_{CDer_{\it{\si,\t}}(\R)}$ and $\vp_\t\p_\t=\rm{id}_{CDer_{\it{\t^{-1}\si}}(\R)}$, which means $\p_\t$ is the inverse of $\vp_\t$. Therefore, $\CDer_{\it{\si,\t}}(\R)$ and $\CDer_{\it{\t^{-1}\si}}(\R)$ are isomorphic as $\mathbb{C}[\partial]$-modules, and we have $\rm{rank}(CDer_{\it{\si, \t}}(\R))= \rm{rank}(CDer_{\it{\t^{-\rm{1}}\si}}(\R))$.
\end{proof}

Proposition \ref{prop:2.1} means that the study of $\CDer_{\it{\si,\t}}(\R)$ with two parameters $\si,\t$ can be reduced to the study of $\CDer_{\it{\si^{'}}}(\R)$ with one parameter $\si^{'}=\t^{-1}\si$. Particularly, if we take $\t=\si$, then $\CDer_{\it{\si,\si}}(\R)$ and $\CDer(\R)$ are isomorphic as $\mathbb{C}[\partial]$-modules. Moreover, we may extend this isomorphic relation to the level of Lie conformal superalgebras.

\begin{prop}\label{prop:2.2}
Suppose that $\R$ is a Lie conformal superalgebra and  $\si$ is an element in $G$. Then $\CDer_{\it{\si, \si}}(\R)$ can be viewed as a Lie conformal superalgebra and $\Der_{\it{\si,\si}}(\R) \cong \CDer(\R)$ as Lie conformal superalgebra.
\end{prop}

\begin{proof}
Define a $\l$-bracket $[-_{\l}-]^{\si}$ on $\CDer_{\it{\si, \si}}(\R)\times\CDer_{\it{\si, \si}}(\R)$ as follow:
\[[f_{\l}g]^{\si} = \vp_{\si}^{-1}([(\vp_{\si}(f))_{\l} (\vp_{\si}(g))]),\quad\forall f, g \in \CDer_{\it{\si, \si}}(\R),\]
where $\varphi_{\sigma}$ is defined as Proposition \ref{prop:2.1}.  It's obvious that $[-_{\l}-]^{\si}$ is well-defined since $\varphi_{\sigma}$ is a bijective map. And it is bilinear because both $\vp_{\si}$ and $\vp_{\si}^{-1}$ are isomorphic as $\mathbb{C}[\partial]$-modules.

For any $f, g$ in $\CDer_{\it{\si, \si}}(\R)$, we have
\begin{align*}
&~~~~[\partial f_{\l}g]^{\si}=\vp_{\si}^{-1}([(\vp_{\si}(\partial f))_{\l} (\vp_{\si}(g))]) \\
&= \vp_{\si}^{-1}([(\vp_{\si}(-\l f))_{\l} (\vp_{\si}(g))]) =-\l\vp_{\si}^{-1}([(\vp_{\si}(f))_{\l} (\vp_{\si}(g))]) = -\l[f_{\l}g]^{\si},
\end{align*}
and
\[[f_{\l}g]^{\si} = \vp_{\si}^{-1}([(\vp_{\si}(f))_{\l} (\vp_{\si}(g))]) = -(-1)^{|f||g|}\vp_{\si}^{-1}([(\vp_{\si}(g))_{-\partial-\l} (\vp_{\si}(f))]) = -(-1)^{|f||g|}[g_{-\partial-\l} f]^{\si}.\]
For any $f, g, h$ in $\CDer_{\it{\si, \si}}(\R)$, we have
\begin{align*}
&~~~~[[f_{\l}g]^{\si}_{\l+\mu}h]^{\si}
=[\vp_{\si}^{-1}([(\vp_{\si}(f))_{\l}(\vp_{\si}(g))])_{\l+\mu} h]^{\si}\\
&=\vp_{\si}^{-1}([\vp_{\si}(\vp_{\si}^{-1}([(\vp_{\si}(f))_{\l} (\vp_{\si}(g))]))_{\l+\mu}(\vp_{\si}(h))])\\
&= \vp_{\si}^{-1}([[(\vp_{\si}(f))_{\l} (\vp_{\si}(g))]_{\l+\mu}(\vp_{\si}(h))]).
\end{align*}
Similarly, we have
\[[f_{\l}[g_{\mu}h]^{\si}]^{\si} = \vp_{\si}^{-1}([(\vp_{\si}(f))_{\l}[(\vp_{\si}(g))_{\mu} (\vp_{\si}(h))]]),\]
\[[g_{\mu}[f_{\l}h]^{\si}]^{\si} = \vp_{\si}^{-1}([(\vp_{\si}(g))_{\mu}[(\vp_{\si}(f))_{\l} (\vp_{\si}(h))]]).\]
So we have
\begin{align*}
&~~~~[[f_{\l}g]^{\si}_{\l+\mu}h]^{\si} - [f_{\l}[g_{\mu}h]^{\si}]^{\si} + (-1)^{|f||g|}[g_{\mu}[f_{\l}h]^{\si}]^{\si}\\
&= \vp_{\si}^{-1}([[(\vp_{\si}(f))_{\l} (\vp_{\si}(g))]_{\l+\mu}(\vp_{\si}(h))]) - \vp_{\si}^{-1}([(\vp_{\si}(f))_{\l}[(\vp_{\si}(g))_{\mu} (\vp_{\si}(h))]])\\
&~~~~+ (-1)^{|f||g|} \vp_{\si}^{-1}([(\vp_{\si}(g))_{\mu}[(\vp_{\si}(f))_{\l} (\vp_{\si}(h))]])\\
&= \vp_{\si}^{-1}([[(\vp_{\si}(f))_{\l} (\vp_{\si}(g))]_{\l+\mu}(\vp_{\si}(h))] - [\vp_{\si}(f)_{\l}[(\vp_{\si}(g))_{\mu} (\vp_{\si}(h))]]\\
&~~~~+(-1)^{|f||g|}[(\vp_{\si}(g))_{\mu}[(\vp_{\si}(f))_{\l} (\vp_{\si}(h))]])\\
&= \varphi_{\sigma}^{-1}(0)=0,
\end{align*}
which means that
$[f_{\l}[g_{\mu}h]^{\si}]^{\si}
=[[f_{\l}g]^{\si}_{\l+\mu}h]^{\si}+
(-1)^{|f||g|}[g_{\mu}[f_{\l}h]^{\si}]^{\si}$.
Hence, $(\CDer_{\it{\si, \si}}(\R), [-_{\l}-]^{\si})$ is a Lie conformal superalgebra.

To complete the proof, it still needs to show that $\vp_{\si}$ is a Lie conformal superalgebras homomorphism. Actually,
$\vp_{\si}([f_{\l}g]^{\si}) = \vp_{\si}(\vp_{\si}^{-1}([(\vp_{\si}(f))_{\l} (\vp_{\si}(g))]))=[(\vp_{\si}(f))_{\l} (\vp_{\si}(g))]$
and $\vp_{\si}$ is a Lie conformal superalgebras homomorphism between $\CDer_{\it{\si, \si}}(\R)$ and $\CDer(\R)$, as desired.
\end{proof}

Recall that the center of a Lie conformal superalgebra $\R$ is the set $\rm{Z}(\R)=\{\it{a}\in\R\;|\;[\it{a}_{\l}\it{b}]=\rm{0},\,\forall \it{b}\in\R\}$, and the centralizer of $a$ in $\R$ is the set $\rm{Z}_{\it{a}}(\R)=\{\it{b}\in\R\;|\;[\it{a}_{\l}\it{b}]=\rm{0}\}$.

\begin{prop}\label{prop:2.3}
Suppose that $\si,\t$ are two elements in $G$ such that $(\si-\t)(\R)\subseteq\rm{Z}(\R)$, then $\CDer_{\it{\si}}(\R) = \CDer_{\it{\t}}(\R)$. In addition, if $(\si - \rm{id}_{\R})(\R) \subseteq \rm{Z}(\R)$, then $\CDer_{\it{\si}}(\R) = \CDer
(\R)$ is a Lie conformal subalgebra of $\Cend(\R)$.
\end{prop}

\begin{proof}
For any  $d$ in $\CDer_{\it{\si}}(\R)$ and
since $(\si - \t)(\R) \subseteq \rm{Z}(\R)$, we have
$$[(d_{\l}(a))_{\l+\mu}((\si-\t)(b))]= 0,\quad\forall a, b\in\R,$$  that is  $[(d_{\l}(a))_{\l+\mu}(\si(b))]=[(d_{\l}(a))_{\l+\mu}(\t(b))]$.
So we can get
\[d_{\l}([a_{\mu}b])
=[(d_{\l}(a))_{\l+\mu}(\si(b))]+(-1)^{|a||d|}[a_{\mu}(d_{\l}(b))] =[(d_{\l}(a))_{\l+\mu}(\t(b))]+(-1)^{|a||d|}[a_{\mu}(d_{\l}(b))] ,\]
which implies that $d \in \CDer_{\it{\t}}(\R)$ and $\CDer_{\it{\si}}(\R) \subseteq \CDer_{\it{\t}}(\R)$.
By switching the roles of $\si$ and $\t$, we have $\CDer_{\it{\t}}(\R) \subseteq \CDer_{\it{\si}}(\R)$. Consequently, we obtain $\CDer_{\it{\si}}(\R) = \CDer_{\it{\t}}(\R)$.

In particular, if we take $\t=\rm{id}_{\R}$, then $\CDer_{\it{\si}}(\R) = \CDer(\R)$ is a Lie conformal subalgebra of $\Cend(\R)$.
\end{proof}

\begin{prop}\label{prop:2.4}
Suppose that $\si,\si^{'}$ are two elements in $G$ such that $f\in\CDer_{\it{\si}}(\R)$ and $g\in\CDer_{\it{\si^{'}}}(\R)$. If $\si$ and $\si^{'}$ commute, $f$ commutes with $\si^{'}$, and $g$ commutes with $\si$, then $[f_{\l}g]_{\mu}
\in\CDer_{\it{\si}\it{\si^{'}}}(\R)$.
\end{prop}
\begin{proof}
For any $a,b\in \R$, we have
\begin{align*}
&~~~~f_{\l}(g_{\mu-\l}([a_{\g}b]))
=f_{\l}([(g_{\mu-\l}(a))_{\mu-\l+\g}(\si^{'}(b))]
+(-1)^{|a||g|}[a_{\g}(g_{\mu-\l}(b))])\\
&=[(f_{\l}(g_{\mu-\l}(a)))_{\mu+\g}(\si(\si^{'}(b)))]
+(-1)^{(|a|+|g|)|f|}[(g_{\mu-\l}(a))_{\mu-\l+\g}(f_{\l}(\si^{'}(b)))]\\
&~~~~+(-1)^{|a||g|}[(f_{\l}(a))_{\l+\g}(\si (g_{\mu-\l}(b)))]
+(-1)^{|a||g|+|a||f|}[a_{\g}(f_{\l}(g_{\mu-\l}(b)))],
\end{align*}
and
\begin{align*}
&~~~~g_{\mu-\l}(f_{\l}([a_{\g}b]))
=g_{\mu-\l}([(f_{\l}(a))_{\l+\g}(\si(b))]
+(-1)^{|a||f|}[a_{\g}(f_{\l}(b))])\\
&=[(g_{\mu-\l}(f_{\l}(a)))_{\mu+\g}(\si^{'}(\si(b)))]
+(-1)^{(|a|+|f|)|g|}[(f_{\l}(a))_{\l+\g}(g_{\mu-\l}(\si(b)))]\\
&~~~~+(-1)^{|a||f|}[(g_{\mu-\l}(a))_{\mu-\l+\g}(\si^{'} (f_{\l}(b)))]
+(-1)^{|a||f|+|a||g|}[a_{\g}(g_{\mu-\l}(f_{\l}(b)))].
\end{align*}
Since $\si$ and $\si^{'}$ commute, $f$ commutes with $\si^{'}$, and $g$ commutes with $\si$, it follows that
\begin{align*}
&~~~~[f_{\l}g]_{\mu}([a_{\g}b])
=(f_{\l}g_{\mu-\l}-(-1)^{|f||g|}g_{\mu-\l}f_{\l})([a_{\g}b])
=f_{\l}(g_{\mu-\l}([a_{\g}b]))-(-1)^{|f||g|}g_{\mu-\l}(f_{\l}([a_{\g}b]))\\
&=[(f_{\l}(g_{\mu-\l}(a)))_{\mu+\g}(\si(\si^{'}(b)))]
+(-1)^{(|a|+|g|)|f|}[(g_{\mu-\l}(a))_{\mu-\l+\g}(f_{\l}(\si^{'}(b)))]\\
&~~~~+(-1)^{|a||g|}[(f_{\l}(a))_{\l+\g}(\si (g_{\mu-\l}(b)))]
+(-1)^{|a||g|+|a||f|}[a_{\g}(f_{\l}(g_{\mu-\l}(b)))]\\
&~~~~-(-1)^{|f||g|}[(g_{\mu-\l}(f_{\l}(a)))_{\mu+\g}(\si^{'}(\si(b)))]
-(-1)^{|a||g|}[(f_{\l}(a))_{\l+\g}(g_{\mu-\l}(\si(b)))]\\
&~~~~-(-1)^{|a||f|+|f||g|}[(g_{\mu-\l}(a))_{\mu-\l+\g}(\si^{'}(f_{\l}(b)))]
-(-1)^{|a||f|+|a||g|+|f||g|}[a_{\g}(g_{\mu-\l}(f_{\l}(b)))]\\
&=[([f_{\l}g]_{\mu}(a))_{\mu+\g}(\si(\si^{'}(b)))]
+(-1)^{|a|(|f|+|g|)}[a_{\g}([f_{\l}g]_{\mu}(b))],
\end{align*}
which shows that $[f_{\l}g]_{\mu}
\in\CDer_{\it{\si}\it{\si^{'}}}(\R)$.
\end{proof}

\begin{cor}\label{corollary:2.5}
Suppose that $\si$ is an involutive automorphism of $G$. If $\si$ commutes with every element of $\CDer_{\it{\si}}(\R)$, then $\CDer_{\it{\si}}(\R)$ is a Lie conformal superalgebra.
\end{cor}
\begin{proof}
Note that $\CDer_{\it{\si}}(\R)$ is a $\mathbb{C}[\partial]$-module, thus it suffices to verify that $\CDer_{\it{\si}}(\R)$ is closed under the $\l$-bracket, that is
\[[f_{\l}g]_{\mu} = f_{\l}g_{\mu-\l}-(-1)^{|f||g|}g_{\mu-\l}f_{\l}\in\CDer_{\it{\si}}(\R), \quad\forall f, g \in \CDer_{\it{\si}}(\R).\]

With a similar discussion as that in the proof of Proposition \ref{prop:2.4}, we have
$$f_{\l}(g_{\mu-\l}([a_{\g}b]))=[([f_{\l}g]_{\mu}(a))_{\mu+\g}\si^{2}(b)]
+(-1)^{|a|(|f|+|g|)}[a_{\g}([f_{\l}g]_{\mu}(b))],\quad\forall a, b \in \R.$$
Since $\si^{2}=\si$, it follows that $[f_{\l}g]_{\mu}\in\CDer_{\rm{\si}}(\R)[\l]$. Therefore, $\CDer_{\it{\si}}(\R)$ is a Lie conformal superalgebra.
\end{proof}

\begin{prop}\label{prop:2.6}
Suppose that $\si, \t$ are two elements in $G$ and $d\in\CDer_{\it{\si}}(\R)$. Then
$\t d\in\CDer_{\it{\t}\it{\si},\it{\t}}(\R)$ and $d\t\in\CDer_{\it{\si}\it{\t},\it{\t}}(\R)$.
\end{prop}
\begin{proof}
For any $a, b \in \R$, we have
\begin{align*}
&~~~~\t d_\l([a_\mu b])=\t([(d_\l(a))_{\l+\mu}(\si(b))]+(-1)^{|a||d|}[a_\mu (d_\l(b))])\\
&=\t([(d_\l(a))_{\l+\mu}(\si(b))])+(-1)^{|a||d|}\t([a_\mu (d_\l(b))])\\
&=[(\t(d_\l(a)))_{\l+\mu}(\t(\si(b)))]+(-1)^{|a||d|}[(\t(a))_\mu (\t(d_\l(b)))],
\end{align*}
which means that
$\t d\in\CDer_{\it{\t}\it{\si},\it{\t}}(\R)$.

Similarly, we can obtain that $d\t\in\CDer_{\it{\si}\it{\t},\it{\t}}(\R)$.
\end{proof}

\begin{prop}\label{prop:2.7}
Suppose that $\si, \t$ are two elements in $G$ such that $c - \si^{-1}\t(c) \notin \rm{Z}_{\it{c}}(\R)$ for any nonzero element $c$ in $\R$. Then $\CDer_{\it{\si}}(\R) \cap \CDer_{\it{\t}}(\R) = \{\rm{0}\}$.
\end{prop}

\begin{proof}
Assume that there exists a nonzero element $d\in\CDer_{\it{\si}}(\R) \cap \CDer_{\it{\t}}(\R)$. Then there exists an element $a_{0} \in \R$ such that $d_{\l}(a_{0}) \neq 0$.
So we have
$$[(d_{\l}(a))_{\l+\mu}(\si(b))]
=[(d_{\l}(a))_{\l+\mu}(\t(b))],\quad\forall a,b \in \R,$$
which implies that
\[[(\si^{-1}(d_{\l}(a)))_{\l+\mu}(b-\si^{-1}\t(b))] = 0.\]
If we take $a=a_{0}$ and $b = b_{0} := \si^{-1}d_{\l}(a_{0})$,
then we can get
\[[b_{0_{\l+\mu}}(b_{0}-\si^{-1}\t(b_{0}))] = 0,\]
which means that $b_{0}-\si^{-1}\t(b_{0}) \in \rm{Z}_{\it{b_{\rm{0}}}}(\R)$. Since $c - \si^{-1}\t(c) \notin \rm{Z}_{\it{c}}(\R)$ with $c\in\R$ and $c\neq 0$, we have $b_{0}=0$. According to the assumption, $d_{\l}(a_{0}) \neq 0$ and $\si^{-1}$ is an isomorphism, we have $b_{0} = \si^{-1}d_{\l}(a_{0}) \neq 0$, which is a contradiction.
\end{proof}

\begin{prop}\label{prop:2.8}
Suppose that $d \in \CDer_{\it{\si}}(\R)$ is an element such that $(d_{\l}\si-\si d_{\l})(\R) \subseteq \rm{Z}(\R)$. Then $[\R_{\l}\R]$ is contained in the kernel of $d_{\l}\si-\si d_{\l}$.
\end{prop}

\begin{proof}
For any $a,b\in\R$, we have
\[d_{\l}(\si([a_{\mu}b]))
=d_{\l}([(\si(a))_{\mu}(\si(b))])
=[(d_{\l}(\si(a)))_{\l+\mu}(\si^{2}(b))]
+(-1)^{|a|+|d|}[(\si(a))_{\mu}(d_{\l}(\si(b)))],\]
and\begin{align*}
&~~~~\si (d_{\l}([a_{\mu}b]))=\si([(d_{\l}(a))_{\l+\mu}(\si(b))])
+(-1)^{|a|+|d|}[a_{\mu}(d_{\l}(b))]\\
&=[(\si (d_{\l}(a)))_{\l+\mu}(\si^{2}(b))]
+(-1)^{|a|+|d|}[(\si(a))_{\mu}(\si (d_{\l}(b)))].
\end{align*}
Since $(d_{\l}\si-\si d_{\l})(\R) \subseteq \rm{Z}(\R)$, we have
\begin{align*}
&~~~~(d_{\l}\si-\si d_{\l})([a_{\mu}b])=d_{\l}\si([a_{\mu}b])-\si d_{\l}([a_{\mu}b])\\ &=[d_{\l}(\si(a))_{\l+\mu}(\si^{2}(b))]
+(-1)^{|a|+|d|}[(\si(a))_{\mu}(d_{\l}(\si(b)))]\\
&~~~~-[(\si(d_{\l}(a)))_{\l+\mu}(\si^{2}(b))]
-(-1)^{|a|+|d|}[(\si(a))_{\mu}(\si(d_{\l}(b)))]\\
&=[((d_{\l}\si-\si d_{\l})(a))_{\l+\mu}(\si^{2}(b))]
+(-1)^{|a|+|d|}[(\si(a))_{\mu}((d_{\l}\si-\si d_{\l})(b))]\\
&=0.
\end{align*}
Consequently, $[\R_{\l}\R]$ is contained in the kernel of $d_{\l}\si-\si d_{\l}$.
\end{proof}

\section{The interiors of conformal $G$-derivations}
\def\theequation{\arabic{section}.\arabic{equation}}
\setcounter{equation} {0}

In this section, we will investigate the structures of $\CDer_{\si}(\R)$ and $\CDer_{G}(\R)$. To understand this, we focus on a special class of $\CDer_{\si}(\R)$ called the \emph{interiors} of $G$-derivations, $\CDer_{G}^{\star}(\R)$. In addition, we study the rationality of the Hilbert series for the direct sum of these interiors of conformal $G$- derivations when $G$ is a cyclic subgroup.

Set
\[\CDer_{\it{\si}}^{+}(\R) = \{d \in \CDer_{\it{\si}}(\R)\;|\;d_{\l}\si = \si d_{\l}\},\;\CDer_{\it{\si}}^{-}(\R) = \{d \in \CDer_{\it{\si}}(\R)\;|\;d_{\l}\t = \t d_{\l},\;\forall \t \in G\}.\]
It is obvious that
$\CDer_{\it{\si}}^{-}(\R) \subseteq \CDer_{\it{\si}}^{+}(\R) \subseteq \CDer_{\it{\si}}(\R)$
and they are all $\mathbb{C}[\partial]$-modules. We now consider some kind of "sum" of them respectively and observe how close from these sum to $\CDer_{G}(\R)$.

Define
\[\CDer_{G}^{+}(\R) := \oplus_{\si \in G}\CDer_{\it{\si}}^{+}(\R),\;\CDer_{G}^{-}(\R) := \oplus_{\si \in G}\CDer_{\it{\si}}^{-}(\R),\]
called the \emph{big interior} and the \emph{small interior} of $\CDer_{G}(\R)$ respectively. Besides, we may define
\[\CDer_{G}^{\star}(\R) := \oplus_{\si \in G}\CDer_{\it{\si}}(\R),\]
called the \emph{interior} of $\CDer_{G}(\R)$. Obviously, we have
$\CDer_{G}^{-}(\R) \subseteq \CDer_{G}^{+}(\R) \subseteq \CDer_{G}^{\star}(\R)$.

\begin{exa}\label{example:3.1}
Let $G = \{\rm{id}_{\R}\}$, the trivial group. Since $\CDer_{G}^{-}(\R) = \CDer_{G}^{+}(\R)$, we have $\CDer_{G}^{-}(\R) = \CDer_{G}^{+}(\R) = \CDer_{G}^{\star}(\R) = \CDer_{G}(\R) = \CDer(\R)$.
\end{exa}

\begin{exa}\label{example:3.2}
 Let $G$ is a cyclic group and $\si$ is its generator. If $* \in \{-, +, \star\}$, then
\[\CDer_{\it{G}}^{*}(\R) = \CDer_{\langle\si\rangle}^{*}(\R) = \oplus_{k \in \mathbb{Z}}\CDer_{\si^{\it{k}}}^{*}(\R)\]
where $\si^{0} = \rm{id}_{\R}$, $\si^{1} = \si$ and $\si^{k} = \si^{k - 1}\si$. For convenience, we denote $\CDer_{\si^{\it{k}}}^{\star}(\R)$ by $\CDer_{\si^{\it{k}}}(\R)$. In this case, $\CDer_{\langle\si\rangle}^{*}(\R)$ is a $\mathbb{Z}$-graded $\mathbb{C}[\partial]$-module and recall that the \emph{Hilbert series} of $\CDer_{\langle\si\rangle}^{*}(\R)$ is defined as follow:
\[H(\CDer_{\langle\si\rangle}^{*}(\R), t) := \sum_{k \in \mathbb{Z}}^{}\rm{rank}(\CDer_{\si^{\it{k}}}^{*}(\R))\it{t}^{\it{k}}.\]
If $\si$ is of finite order, then $H(\CDer_{\langle\si\rangle}^{*}(\R), t)$ is a polynomial function in $\mathbb{Z}[t]$.
\end{exa}

\begin{prop}\label{prop:3.3}
Suppose that $G$ is an abelian group. Then $\CDer_{G}^{-}(\R)$ is a Lie conformal superalgebra with the $\l$-bracket $[-_{\l}-]$.
\end{prop}

\begin{proof}
Since $\CDer_{G}^{-}(\R)$ is a $\mathbb{C}[\partial]$-module, it is sufficient to show that $\CDer_{G}^{-}(\R)$ is closed under the $\l$-bracket. For any $f\in\CDer_{\si}^{-}(\R)$ and $g\in\CDer_{\t}^{-}(\R)$ with $\si,\t\in G$ and $a,b\in \R$, we have
\begin{align*}
&~~~~f_{\l}(g_{\mu-\l}([a_{\g}b]))
=f_{\l}([(g_{\mu-\l}(a))_{\mu-\l+\g}(\t(b))]
+(-1)^{|a||g|}[a_{\g}(g_{\mu-\l}(b))])\\
&=[(f_{\l}(g_{\mu-\l}(a)))_{\mu+\g}(\si(\t(b)))]
+(-1)^{(|a|+|g|)|f|}[(g_{\mu-\l}(a))_{\mu-\l+\g}(f_{\l}(\t(b)))]\\
&~~~~+(-1)^{|a||g|}[(f_{\l}(a))_{\l+\g}(\si (g_{\mu-\l}(b)))]
+(-1)^{|a||g|+|a||f|}[a_{\g}(f_{\l}(g_{\mu-\l}(b)))],
\end{align*}
and
\begin{align*}
&~~~~g_{\mu-\l}(f_{\l}([a_{\g}b]))
=g_{\mu-\l}([(f_{\l}(a))_{\l+\g}(\si(b))]
+(-1)^{|a||f|}[a_{\g}(f_{\l}(b))])\\
&=[(g_{\mu-\l}(f_{\l}(a)))_{\mu+\g}(\t(\si(b)))]
+(-1)^{(|a|+|f|)|g|}[(f_{\l}(a))_{\l+\g}(g_{\mu-\l}(\si(b)))]\\
&~~~~+(-1)^{|a||f|}[(g_{\mu-\l}(a))_{\mu-\l+\g}(\t (f_{\l}(b)))]
+(-1)^{|a||f|+|a||g|}[a_{\g}(g_{\mu-\l}(f_{\l}(b)))].
\end{align*}
Since $f_{\l}\t=\t f_{\l}$, $g_{\mu-\l}\si=\si g_{\mu-\l}$ and $G$ is abelian, we have
\begin{align*}
&~~~~[f_{\l}g]_{\mu}([a_{\g}b])
=(f_{\l}g_{\mu-\l}-(-1)^{|f||g|}g_{\mu-\l}f_{\l})([a_{\g}b])\\
&=[(f_{\l}(g_{\mu-\l}(a)))_{\mu+\g}(\si(\t(b)))]
+(-1)^{(|a|+|g|)|f|}[(g_{\mu-\l}(a))_{\mu-\l+\g}(f_{\l}(\t(b)))]\\
&~~~~+(-1)^{|a||g|}[(f_{\l}(a))_{\l+\g}(\si (g_{\mu-\l}(b)))]
+(-1)^{|a||g|+|a||f|}[a_{\g}(f_{\l}(g_{\mu-\l}(b)))]\\
&~~~~-(-1)^{|f||g|}[(g_{\mu-\l}(f_{\l}(a)))_{\mu+\g}(\t(\si(b)))]
-(-1)^{|a||g|}[(f_{\l}(a))_{\l+\g}(g_{\mu-\l}(\si(b)))]\\
&~~~~-(-1)^{|a||f|+|f||g|}[(g_{\mu-\l}(a))_{\mu-\l+\g}(\t (f_{\l}(b)))]
-(-1)^{|f||g|+|a||f|+|a||g|}[a_{\g}(g_{\mu-\l}(f_{\l}(b)))].\\
&=[(f_{\l}g_{\mu-\l}-(-1)^{|f||g|}g_{\mu-\l}f_{\l})(a)_{\mu+\g}(\si(\t(b)))]
+[a_{\g}((f_{\l}g_{\mu-\l}-(-1)^{|f||g|}g_{\mu-\l}f_{\l})(b))]\\
&=[([f_{\l}g]_{\mu}(a))_{\mu+\g}(\si(\t(b)))]+[a_{\g}([f_{\l}g]_{\mu}(b))],
\end{align*}
which implies that $[f_{\l}g]_{\mu}\in\CDer_{\si\t}(\R)[\l]$. Obviously, $[f_{\l}g]_{\mu}$ commutes with every element in $G$, and so $[f_{\l}g]_{\mu}\in\CDer_{\si\t}^{-}(\R)[\l]\subseteq\CDer_{G}^{-}(\R)[\l]$. Consequently, $\CDer_{G}^{-}(\R)$ is a Lie conformal superalgebra.
\end{proof}


According to the above results, we can see that $\CDer_{G}(\R)$ may be very large and complicated. In the following, we will focus on the interiors of $\CDer_{G}(\R)$ where $G$ is an infinite cyclic group. Particularly, we will investigate the
important invariant, the Hilbert series, which encodes the ranks of submodules into an infinite series.

\begin{prop}\label{prop:3.4}
Let $G=\langle\si\rangle$ be an infinite cyclic group. If there exists $l_{0} \in \mathbb{N}^{+}$ and $d \in \CDer_{\si^{\it{l_{\rm{0}}}}}(\R)$ such that $\f_{d}$ is invertible restricted to $\CDer_{\si^{\it{i}}}(\R)$ for all $i \in \mathbb{Z}\setminus\{l_{0}\}$, then $H(\CDer_{\it{G}}^{-}(\R), t)$ is a rational function.
\end{prop}

\begin{proof}
Since $G$ is an infinite cyclic group generated by $\si$, we have $\f_{d} : \CDer_{\si^{\it{k}}}^{-}(\R) \rightarrow \CDer_{\si^{\it{k + l_{\rm{0}}}}}^{-}(\R)$ is a $\mathbb{C}[\partial]$-module isomorphism for all $k \in \mathbb{Z}\setminus\{l_{0}\}$ by Proposition \ref{prop:3.3}. Hence, $\rm{rank}(\CDer_{\si^{\it{k}}}^{-}(\R)) = \rm{rank}(\CDer_{\si^{\it{k + l_{\rm{0}}}}}^{-}(\R)) = \rm{rank}(\CDer_{\si^{\it{k - l_{\rm{0}}}}}^{-}(\R))$ for each $k \in \mathbb{N}\setminus\{l_{0}\}$. It is obvious that,
\[H(\CDer_{\it{G}}^{-}(\R), t) = \sum_{k \in \mathbb{Z}}\rm{rank}(\CDer_{\si^{\it{k}}}^{-}(\R))t^{k} = \sum_{\it{k} = \it{l}_{\rm{0}} + \rm{1}}^{\infty}\rm{rank}(\CDer_{\si^{\it{k}}}^{-}(\R))\it{t}^{\it{k}} + \sum_{-\infty}^{k = \it{l}_{\rm{0}}}\rm{rank}(\CDer_{\si^{\it{k}}}^{-}(\R))\it{t}^{\it{k}}.\]
In addition,
\begin{align*}
&~~~~\sum_{k = l_{\rm{0}} + \rm{1}}^{\infty}\rm{rank}(\CDer_{\si^{\it{k}}}^{-}(\R))\it{t}^{\it{k}}\\
&= (m_{0}t^{l_{0} + 1} + m_{1}t^{l_{0} + 2} + \cdots + m_{l_{0} - 1}t^{2l_{0}}) + (m_{0}t^{2l_{0} + 1} + m_{1}t^{2l_{0} + 2} + \cdots + m_{l_{0} - 1}t^{3l_{0}}) + \cdots\\
&= (m_{0} + m_{1}t + \cdots + m_{l_{0} - 1}t^{l_{0} - 1})t^{l_{0} + 1} + (m_{0} + m_{1}t + \cdots + m_{l_{0} - 1}t^{l_{0} - 1})t^{2l_{0} + 1} + \cdots\\
&= t^{l_{0} + 1}(m_{0} + m_{1}t + \cdots + m_{l_{0} - 1}t^{l_{0} - 1})(1 + t^{l_{0}} + t^{2l_{0}} + \cdots)\\
&= t^{l_{0} + 1}(m_{0} + m_{1}t + \cdots + m_{l_{0} - 1}t^{l_{0} - 1})\frac{1}{1 - t^{l_{0}}}\\
&= \frac{t^{l_{0} + 1}\sum_{i = 0}^{l_{0} - 1}m_{i}t^{i}}{1 - t^{l_{0}}}
\end{align*}
where $m_{i} = \rm{rank}(\CDer_{\si^{\it{l_{\rm{0}} + \rm{1} + i}}}^{-}(\R))$ for $0 \leq i \leq l_{0} - 1$. Similarly,
\begin{align*}
&~~~~\sum_{-\infty}^{k = l_{\rm{0}}}\rm{rank}(\CDer_{\si^{\it{k}}}^{-}(\R))\it{t}^{\it{k}}\\
&= \sum_{k = 1}^{l_{0}}\rm{rank}(\CDer_{\si^{\it{k}}}^{-}(\R))\it{t}^{\it{k}} + \sum_{-\infty}^{k = \rm{0}}\rm{rank}(\CDer_{\si^{\it{k}}}^{-}(\R))\it{t}^{\it{k}}\\
&= \sum_{k = 1}^{l_{0}}\rm{rank}(\CDer_{\si^{\it{k}}}^{-}(\R))\it{t}^{\it{k}} + \sum_{k = \rm{0}}^{\infty}\rm{rank}(\CDer_{\si^{\it{-k}}}^{-}(\R))\it{t}^{\it{k}}\\
&= m_{0}t + \cdots + m_{l_{0} - 1}t^{l_{0}} + (m_{0}t^{-l_{0} + 1} + m_{1}t^{-l_{0} + 2} + \cdots + m_{l_{0} - 1})\\
&~~~~+(m_{0}t^{-2l_{0} + 1} + m_{1}t^{-2l_{0} + 2} + \cdots + m_{l_{0} - 1}t^{-l_{0}}) + \cdots\\
&= m_{0}t + \cdots + m_{l_{0} - 1}t^{l_{0}} + (m_{0}t^{-l_{0} + 1} + m_{1}t^{-l_{0} + 2} + \cdots + m_{l_{0} - 1})(1 + t^{-l_{0}} + \cdots)\\
&= (m_{0} + \cdots + m_{l_{0} - 1}t^{l_{0} - 1})t + \frac{\sum_{i = 0}^{l_{0} - 1}m_{i}t^{-(l_{0} - 1 - i)}}{1 - t^{-l_{0}}}.
\end{align*}
Therefore, we have
\begin{align*}
&~~~~H(\CDer_{\it{G}}^{-}(\R), t)\\
&= \frac{t^{l_{0} + 1}\sum_{i = 0}^{l_{0} - 1}m_{i}t^{i}}{1 - t^{l_{0}}} + (m_{0} + \cdots + m_{l_{0} - 1}t^{l_{0} - 1})t + \frac{\sum_{i = 0}^{l_{0} - 1}m_{i}t^{-(l_{0} - 1 - i)}}{1 - t^{-l_{0}}}\\
&= \frac{t}{1 - t^{l_{0}}}\sum_{i = 0}^{l_{0} - 1}m_{i}t^{i} + \frac{1}{1 - t^{-{l_{0}}}}\sum_{i = 0}^{l_{0} - 1}m_{i}t^{-(l_{0} - 1 - i)}.
\end{align*}
Consequently, $H(\CDer_{\it{G}}^{-}(\R), t)$ is a rational function.
\end{proof}

\section{Applications}
\def\theequation{\arabic{section}.\arabic{equation}}
\setcounter{equation} {0}

In this section, we study the relation between conformal $(\si,\t)$-derivation and some (generalized) conformal derivations of a Lie conformal superalgebra $\R$, such as centroids and conformal $(\a,\b,\g)$-derivations.
\subsection{Relation with centroids}
Recall that $\rm{ad}:\R\it{\ra} \Cend(\R)$ denotes the adjoint map sending $a$ to $\rm{ad}(\it{a})$ with $\rm{ad}(\it{a})_{\it{\l}}(\it{b})=[a_{\l}b]$, where $a,b\in\R$. And we denote $\rm{ad}(\R)=\{\rm{ad}(\it{a})|a\in\R\}$.

\begin{prop}\label{prop:4.1}
Suppose that $\si \in G$ and $d \in \C(\R) \cap \CDer_{\si}(\R)$. Then $\rm{ad}(\it{d_{\l}(a)})_{\it{\mu}} = \rm{0}$ for any $a\in\R$. In addition, if $\rm{Z}(\R)= \{\rm{0}\}$, then $\C(\R) \cap \CDer_{\si}(\R) = \{\rm{0}\}$.
\end{prop}

\begin{proof}
For any $d \in \C(\R) \cap \CDer_{\si}(\R)$ and $a,b\in\R$, we have
$$d_{\l}([a_{\g}b])
=[(d_{\l}(a))_{\l+\g}(\si(b))]+(-1)^{|a||d|}[a_{\g}(d_{\l}(b))],$$
and
$$(-1)^{|a||d|}[a_{\g}(d_{\l}(b))]=d_{\l}([a_{\l}b]),$$
which implies
$$[(d_{\l}(a))_{\l+\g}(\si(b))]=d_{\l}([a_{\g}b])-(-1)^{|a||d|}[a_{\g}(d_{\l}(b))]=0.$$
Since $\si$ is a bijective map, we can obtain
$d_{\l}(a)\in \rm{Z}(\R)[\l] = \rm{Ker}(\rm{ad})[\l]$. Hence, $\rm{ad}(\it{d_{\l}(a)})_{\it{\mu}} =\rm{0}$ for any $a\in\R$.

Particularly, if $\rm{Z}(\R)= \{\rm{0}\}$, then  $d=0$. Thus, $\C(\R) \cap \CDer_{\si}(\R) = \{\rm{0}\}$.
\end{proof}

\begin{lem}\label{lemma:4.2}
Suppose that $\si \in G$ and $d \in \CDer_{\si}(\R)$. Then for any $a\in\R$, we have $$[d_{\l}(\rm{ad}(\it{a}))]_{\mu} = \si\rm{ad}(\it{\si^{-1}}\it{d}_{\it{\l}}(\it{a}))_{\it{\mu}}.$$
\end{lem}

\begin{proof}
For any $b \in \R$, we have
\begin{align*}
&~~~~[d_{\l}(\rm{ad}(\it{a}))]_{\mu}(b)
=(d_{\l}(\rm{ad}(\it{a}))_{\mu-\l}-(-1)^{|a||d|}(\rm{ad}(\it{a}))_{\mu-\l}d_{\l})(b)\\ &=d_{\l}((\rm{ad}(\it{a}))_{\mu-\l}(b))-(-1)^{|a||d|}(\rm{ad}(\it{a}))_{\mu-\l}(d_{\l}(b))
=d_{\l}([a_{\mu-\l}b])-(-1)^{|a||d|}[a_{\mu-\l}(d_{\l}(b))]\\
&=[(d_{\l}(a))_{\mu}(\si(b))]
=\si([(\si^{-1}(d_{\l}(a)))_{\mu}b])
=\si(\rm{ad}(\it{\si^{-1}}\it{d}_{\it{\l}}(\it{a}))_{\it{\mu}}(b)).
\end{align*}
Therefore, $[d_{\l}\rm{ad}(\it{a})]_{\mu} = \si\rm{ad}(\it{\si^{-1}}\it{d}_{\it{\l}}(\it{a}))_{\it{\mu}}$.
\end{proof}

\begin{lem}\label{lemma:4.3}
Let $a \in \R$ and $\si \in G$. Define a map $\f_{a}^{\si} : \CDer_{\si}(\R) \ra \rm{ad}(\R)_{\g}$, given by $d\mapsto \rm{ad}(\si^{-1}\it{d}_{\l}(a))$.
Then $\f_{a}^{\si}$ is a $\mathbb{C}[\partial]$-module homomorphism.
\end{lem}

\begin{proof}
For any $f, g \in \CDer_{\si}(\R)$ and $b \in \R$, we have
\begin{align*}
&~~~~(\f_{a}^{\si}(f_{\l} + g_{\mu}))(b)
= \rm{ad}(\si^{-1}((\it{f}_{\l} +\it{g}_{\mu})(a)))_{\g}(b)
=[(\si^{-1}((f_{\l} + g_{\mu})(a)))_{\g}b] \\
&=\si^{-1}([((f_{\l} + g_{\mu})(a))_{\g}(\si(b))])
=\si^{-1}([(f_{\l}(a))_{\g}(\si(b))] + [(g_{\mu}(a))_{\g}(\si(b))])\\
&= \si^{-1}([(f_{\l}(a))_{\g}(\si(b))]) + \si^{-1}([(g_{\mu}(a))_{\g}(\si(b))])
=[(\si^{-1}(f_{\l}(a)))_{\g}b]+
[(\si^{-1}(g_{\mu}(a)))_{\g}b]\\
&=\rm{ad}(\si^{-1}(\it{f}_{\l}(\it{a})))_{\g}(\it{b})
+ \rm{ad}(\si^{-1}(\it{g}_{\mu}(\it{a})))_{\g}(\it{b})
= (\f_{a}^{\si}(f_{\l}) + \f_{a}^{\si}(g_{\mu}))_{\g}(b),
\end{align*}
and
\begin{align*}
&~~~~(\f_{a}^{\si}(\partial f_{\l}))_{\g}(b)
=\rm{ad}(\si^{-1}(\partial\it{f}_{\l})(a))_{\g}(b)
=\rm{ad}(\si^{-1}(-\l\it{f}_{\l})(a))_{\g}(b)\\
&=\rm{ad}(-\l\si^{-1}\it{f}_{\l}(a))_{\g}(b)
=[-\l\si^{-1}\it{f}_{\l}(a)_{\g}b]
=-\l[\si^{-1}\it{f}_{\l}(a)_{\g}b]\\
&=-\l\rm{ad}(\si^{-1}(\it{f}_{\l}(a)))_{\g}(b)
=\partial\rm{ad}(\si^{-1}(\it{f}_{\l}(a)))_{\g}(b)
=(\partial(\f_{a}^{\si}(f_{\l})))_{\g}(b).
\end{align*}
Therefore, $\f_{x}^{\si}$ is a $\mathbb{C}[\partial]$-module homomorphism.
\end{proof}

\begin{prop}\label{prop:4.4}
If $a \in \R$ and $\si \in G$, then
\[\rm{Ker}(\it{\f_{a}^{\si}}) = \{d \in \CDer_{\si}(\R)\;|\; d_{\l}(a) \in \rm{Z}(\R)[\l]\}.\]
What's more, $\rm{Ker}(\it{\f_{a}^{\si}})$ is a subalgebra of $\Cend(\R)$.
\end{prop}

\begin{proof}
According to Lemma \ref{lemma:4.2}, we have
\begin{align*}
\rm{Ker}(\it{\f_{a}^{\si}})
&=\{d\in\CDer_{\si}(\R)\;|\;\rm{ad}(\si^{-1}(\it{d}_{\l}(a)))_{\mu}(b)
=\rm{0},\;\forall \it{b \in \R}\}\\
&=\{d\in\CDer_{\si}(\R)\;|\;\si(\rm{ad}(\si^{-1}\it{d}_{\l}(a))_{\mu}(b))=\rm{0},\;\forall \it{b \in \R}\}\\
&=\{d\in\CDer_{\si}(\R)\;|\;[d_{\l}(\rm{ad}(\it{a}))]_{\mu}(b)=\rm{0},\;\forall \it{b \in \R}\}\\
&=\{d\in\CDer_{\si}(\R)\;|\;[(d_{\l}(a))_{\mu}(\si(b))]=\rm{0},\;\forall \it{b \in \R}\}\\
&=\{d\in\CDer_{\si}(\R)\;|\;[(d_{\l}(a))_{\mu}b]=\rm{0},\;\forall \it{b \in \R}\}\\
&=\{d\in\CDer_{\si}(\R)\;|\;d_{\l}(a)\in\rm{Z}(\R)[\l]\}.
\end{align*}

Moreover, it is obvious that $\rm{Ker}(\it{\f_{a}^{\si}})$ is a $\mathbb{C}[\partial]$-module. We only need to show $\rm{Ker}(\it{\f_{a}^{\si}})$ is a Lie conformal superalgebra. For any $f,g\in\rm{Ker}(\it{\f_{a}^{\si}})$ and $b\in\R$, we have
\begin{align*}
&~~~~[([f_{\l}g]_{\mu}(a))_{\g}(\si(b))]
=[(f_{\l}(g_{\mu-\l}(a)))_{\g}(\si(b))]-(-1)^{|f||g|}
[(g_{\mu-\l}(f_{\l}(a)))_{\g}(\si(b))]\\
&=f_{\l}([(g_{\mu-\l}(a))_{\g-\l}b])-(-1)^{(|a|+|g|)|f|}
[(g_{\mu-\l}(a))_{\g-\l}(f_{\l}(b))]\\
&~~~~-(-1)^{|f||g|}g_{\mu-\l}([(f_{\l}(a))_{\g-\mu+\l}b])
+(-1)^{(|a|+|f|)|g|+|f||g|}[(f_{\l}(a))_{\g-\mu+\l}(g_{\mu-\l}(b))]\\
&= 0.
\end{align*}
Since $\si$ is an isomorphism, it follows that $[f_{\l}g]_{\mu}(a)\in\rm{Z}(\R)[\l]$, which implies that $[f_{\l}g]_{\mu}(a)\in \rm{Ker}(\it{\f_{a}^{\si}})[\l]$. Therefore, $\rm{Ker}(\it{\f_{x}^{\si}})$ is a subalgebra of $\Cend(\R)$.
\end{proof}

As a corollary, we can obtain the following main result.
\begin{cor}\label{cor:4.5}
Suppose that $\R$ be a centerless Lie conformal superalgebra. If there exists an element $a_{0} \in \R$ such that $d_{\l}(a_{0}) \neq 0$ for all $d \in \CDer_{\si}(\R)$, then $\rm{rank}(\CDer_{\si}(\R)) \leq \rm{rank}(\R)$.
\end{cor}

\begin{proof}
Note that $\rm{ad} : \R \rightarrow \rm{ad}(\R)$ is an isomorphism. Since $Z(\R)=\{0\}$, we have  $\f_{a_{0}}^{\si}$ is injective by Proposition \ref{prop:4.4}.  Hence, as a $\mathbb{C}[\partial]$-module, $\CDer_{\si}(\R)$ can be embedded into $\R$.
\end{proof}

\subsection{Relation with conformal $(\a,\b, \g)$-derivations}

Similarly to the investigation of conformal $(\a,\b, \g)$-derivations on Lie conformal algebra, we shall define it on Lie conformal superalgebra.

Let $\R$ be an arbitrary Lie conformal superalgebra. We call a conformal linear map $d\in \Cend (\R)$ a \emph {conformal $(\alpha,\beta,\gamma)$-derivation} of degree $\th$ of $\R$ if there exist $\a,\b,\g\in\mathbb{C}$ such that for any $a,b\in \R$, the following relation is satisfied:
\begin{equation*}
\alpha d_\lambda([a_\mu b])=[(\beta d_\l (a))_{\l+\mu}b]+(-1)^{|a||\th|}[a_\mu(\g d_\l(b))].
\end{equation*}
For any given $\alpha,\beta,\gamma\in \mathbb{C}$, we denote the set of all conformal $(\alpha,\beta,\gamma)$-derivations of degree $\th$ by $\CDer_{(\alpha,\beta,\gamma)}(\R)_{\th}$, i.e.
\begin{equation*}\label{40000}
\CDer_{(\alpha,\beta,\gamma)}(\R)_{\th}=\{d\in \Cend(\R)  \,~|~\, \alpha d_\l([a_\mu b])=[(\beta d_\l (a))_{\l+\mu}b]+(-1)^{|a||\th|}[a_\mu(\g d_\l (b))],\;\forall~a,b\in \R\}.
\end{equation*}
Denote by $\CDer_{(\alpha,\beta,\gamma)}(\R)
=\CDer_{(\alpha,\beta,\gamma)}(\R)_{\bar{0}}
\oplus \CDer_{(\alpha,\beta,\gamma)}(\R)_{\bar{1}}$ the set of all conformal $(\alpha,\beta,\gamma)$-derivations of $\R$.

Obviously, we have the following conclusions.
\begin{prop}\label{prop:4.6}
\begin{itemize}\parskip-3pt
\item[\rm(1)]$\CDer_{(1,1,1)}(\R)=\CDer(\R)$.
\item[\rm(2)]$\CDer_{(0,1,-1)}(\R)=\QC(\R)$.
\item[\rm(3)]$\CDer_{(1,0,0)}(\R)\cap \CDer_{(0,1,0)}(\R)=\ZDer(\R)$.
\item[\rm(4)]$\CDer_{(\alpha,\beta,\gamma)}(\R)=
\CDer_{(\triangle\alpha,\triangle\beta,\triangle\gamma)}(\R)$,\ \ \mbox{ \ for any $\triangle\in \mathbb{C}^*$.}
\end{itemize}
\end{prop}

Furthermore, the following proposition has many useful applications.

\begin{prop}\label{prop:4.7}
$\CDer_{(\alpha,\beta,\gamma)}(\R)=\CDer_{(0,\beta-\gamma,\gamma-\beta)}(\R)\cap \CDer_{(2\alpha,\beta+\gamma,\beta+\gamma)}(\R)$, where $\alpha,\beta,\gamma\in \mathbb{C}$.
\end{prop}

\begin{proof}
For any $d\in\CDer_{(\alpha,\beta,\gamma)}(\R)$ and $a,b\in\R$, we get the following equality:
\begin{equation}\label{4001}
\alpha d_\lambda([a_\mu b])=[(\beta d_\lambda (a))_{\lambda+\mu}b]+(-1)^{|a||d|}[a_\mu(\gamma d_\lambda (b))].
\end{equation}
By the skew symmetry and (\ref{4001}), for any $a,b\in\R$, we obtain
\begin{equation}\label{4002}
(-1)^{|a||b|}\alpha d_\lambda([b_{-\mu-\partial} a])=(-1)^{(|a|+|d|)|b|}[b_{-\lambda-\mu-\partial} (\beta d_\lambda (a))]+(-1)^{|a||d|+(|b|+|d|)|a|}[(\gamma d_\lambda (b))_{-\mu-\partial}a].
\end{equation}
By the conformal sesquilinearity and replacing $-\l-\mu-\partial$ by $\mu^{'}$ in (\ref{4002}), we get
\begin{equation}\label{4003}
\alpha d_\lambda([b_{\mu^{'}} a])=(-1)^{(|b||d|}[b_{\mu^{'}} (\beta d_\lambda (a))]+[(\gamma d_\lambda (b))_{\l+\mu^{'}} a].
\end{equation}
Then by changing the place of $a,b$ and replacing $\mu^{'}$ by $\mu$ in (\ref{4003}), we have
\begin{equation}\label{4004}
\alpha d_\lambda([a_{\mu} b])=(-1)^{(|a||d|}[a_{\mu} (\beta d_\lambda (b))]+[(\gamma d_\lambda (a))_{\l+\mu} b].
\end{equation}
By (\ref{4001}) and (\ref{4004}), we can get
\begin{equation}\label{4005}
0=(\b-\g)([( d_\lambda (a))_{\l+\mu} b]-(-1)^{(|a||d|}[a_{\mu} ( d_\lambda (b))]),
\end{equation}
and
\begin{equation}\label{4006}
2\alpha d_\lambda([a_{\mu} b])=
(\b+\g)([( d_\lambda (a))_{\l+\mu} b]+(-1)^{(|a||d|}[a_{\mu} ( d_\lambda (b))]).
\end{equation}

Therefore, $\CDer_{(\alpha,\beta,\gamma)}(\R)\subseteq \CDer_{(0,\beta-\gamma,\gamma-\beta)}(\R)\cap \CDer_{(2\alpha,\beta+\gamma,\beta+\gamma)}(\R)$.
Similarly, starting with the equation (\ref{4005}) and (\ref{4006}), we can also conclude that (\ref{4001}) holds.

Hence,
$\CDer_{(\alpha,\beta,\gamma)}(\R)=\CDer_{(0,\beta-\gamma,\gamma-\beta)}(\R)\cap \CDer_{(2\alpha,\beta+\gamma,\beta+\gamma)}(\R)$.
\end{proof}

Using similar methods in \cite{FHS2019} and \cite{NH2008}, we have the following theorem.

\begin{thm}\label{theorem:4.8}
For any $\alpha,\beta,\gamma\in \mathbb{C}$, there exists $\delta\in\mathbb{C}$ such that the subspace $\CDer_{(\alpha,\beta,\gamma)}(\R)\subseteq \Cend (\R)$ is equal to one of the following four cases:~\rm{(1)}~$\CDer_{(\delta,0,0)}(\R)$, \rm{(2)}~$\CDer_{(\delta,1,-1)}(\R)$, \rm{(3)}~$\CDer_{(\delta,1,0)}(\R)$, \rm{(4)}~$\CDer_{(\delta,1,1)}(\R)$.
\end{thm}

Similarly, we can get the main result of this section.

\begin{thm}\label{theorem:4.9}
For any $\alpha,\beta,\gamma\in \mathbb{C}$, then $\CDer_{(\alpha,\beta,\gamma)}(\R)$ is equal to one of the following subspaces of $\Cend(\R)$:
\begin{itemize}\parskip-3pt
\item[\rm(i)]$\CDer_{(0,0,0)}(\R)=\Cend(\R)$.
\item[\rm(ii)]$\CDer_{(1,0,0)}(\R)=\{d_\lambda\in \Cend(\R)  \,~|~\,  d_\lambda([R,R])=0\}$.
\item[\rm(iii)]$\CDer_{(0,1,-1)}(\R)=\QC(\R)$.
\item[\rm(iv)]$\CDer_{(\delta,1,-1)}(\R)=\CDer_{(0,1,-1)}(\R)\cap \CDer_{(1,0,0)}(\R)$.
\item[\rm(v)]$\CDer_{(\delta,1,1)}(\R)$, $\delta\in\mathbb{C}$.
\item[\rm(vi)]$\CDer_{(\delta,1,0)}(\R)=\CDer_{(0,1,-1)}(\R)\cap \CDer_{(2\delta,1,1)}(\R)$.
\end{itemize}
\end{thm}

We turn now to the problem of relation between conformal $(\si,\t)$-derivation and conformal $(\a,\b, \g)$-derivations.

\begin{lem}\label{lemma:4.10}
For any $\alpha,\beta,\gamma\in \mathbb{C}$, we have
$\CDer_{(\alpha,\beta,\gamma)}(\R)
=\CDer_{(\frac{\a}{\b+\l},1,0)}(\R)$.
\end{lem}
\begin{proof}
According to Proposition \ref{prop:4.7}, we have
\begin{align*}
&~~~~\CDer_{(\alpha,\beta,\gamma)}(\R)\\
&=\CDer_{(0,\beta-\gamma,\gamma-\beta)}(\R)\cap \CDer_{(2\alpha,\beta+\gamma,\beta+\gamma)}(\R)\\
&=\CDer_{(0,1,-1)}(\R)\cap \CDer_{(\frac{2\alpha}{\beta+\gamma},1,1)}(\R)\\
&=\CDer_{(\frac{\a}{\b+\l},1,0)}(\R).
\end{align*}
This lemma is proved.
\end{proof}

\begin{prop}\label{prop:4.11}
Suppose that $\si\in G$ and $\a\in\mathbb{C}$ such that $(\si-\a\rm{id}_{\R})(\R)\subseteq \rm{Z}(\R)$. If $\a\neq1$, then $\CDer_{\si}(\R)=\CDer_{(\frac{1}{\a+1},1,0)}(\R)$;
if $\a=1$, then $\CDer_{\si}(\R)=\CDer_{(1,1,1)}(\R)=\CDer(\R)$.
\end{prop}

\begin{proof}
Assume that $d\in\CDer_{\si}(\R)$.
Since $(\si-\a\rm{id}_{\R})(\R)\subseteq \rm{Z}(\R)$, we have
$$[(d_{\l}(a))_{\mu}(\si(b))]=[(d_{\l}(a))_{\mu}(\a b)],\quad \forall a,b\in\R.$$
Notice that
$$d\in\CDer_{\si}(\R)\Leftrightarrow d_{\l}([(a_{\mu}b])=[(d_{\l}(a))_{\mu}(\a b)]+(-1)^{|a||d|}[a_{\mu}(d_{\l}(b))]\Leftrightarrow d\in\CDer_{(1,\a,1)}(\R),$$
which implies that $\CDer_{\si}(\R)=\CDer_{(1,\a,1)}(\R)$. If $\a\neq1$, then $\CDer_{(1,\a,1)}(\R)=\CDer_{(\frac{1}{\a+1},1,0)}(\R)$ by Lemma \ref{lemma:4.10}. Therefore, $\CDer_{\si}(\R)=\CDer_{(\frac{1}{\a+1},1,0)}(\R)$.
If $\a=1$, then $\CDer_{\si}(\R)=\CDer_{(1,1,1)}(\R)=\CDer(\R)$.
\end{proof}

\begin{prop}\label{prop:4.12}
If $\d\neq 0$, then
$\CDer_{(\d,1,-1)}(\R)
=\CDer_{\rm{id}_\R,-\rm{id}_\R}(\R)$ and $\CDer_{(\d,1,1)}(\R)
=\CDer_{\frac{1}{\d}\rm{id}_\R,\frac{1}{\d}\rm{id}_\R}(\R)$.
\end{prop}
\begin{proof}
According to Proposition \ref{prop:4.7}, we have
\begin{align*}
&~~~~\CDer_{(\d,1,1)}(\R)\\
&=\CDer_{(0,2,-2)}(\R)\cap \CDer_{(2\d,0,0)}(\R)\\
&=\CDer_{(0,2,-2)}(\R)\cap \CDer_{(2,0,0)}(\R)\\
&=\CDer_{(1,1,-1)}(\R).
\end{align*}
For any $d\in\CDer_{(1,1,-1)}(\R)$, we have
$$d_\l([a_\mu b])=[(d_\l(a))_{\l+\mu}b]-(-1)^{|a||d|}[a_\mu (d_\l(b))],$$
which means $d$ is a $(\si,\t)$-derivation with $\si=\rm{id}_{\R}$ and $\t=-\rm{id}_{\R}$.
Therefore, $\CDer_{(\d,1,-1)}(\R)
=\CDer_{\rm{id}_\R,-\rm{id}_\R}(\R)$.

Similarly, $\CDer_{(\d,1,1)}(\R)$ is a $(\si,\t)$-derivation with $\si=\frac{1}{\d}\rm{id}_{\R}$ and $\t=\frac{1}{\d}\rm{id}_{\R}$.
\end{proof}

\renewcommand{\refname}{References}

\end{document}